\documentclass[12pt]{article}
\usepackage{amsmath,amssymb,amsthm,graphics,color}
\title{Discrete Low-Discrepancy Sequences}
\author{{\ \ \ } \and Omer Angel \and Alexander E. Holroyd \and {\ \ \ \ }
\and James B.\ Martin \and James Propp}
\date{2 July 2010}

\newtheorem{thm}{Theorem}
\newtheorem{cor}[thm]{Corollary}
\newtheorem*{thm1'}{Theorem 1$'$}

\theoremstyle{definition}

\newtheorem*{remark*}{Remark}

\newcommand{\calS}{{\cal S}}
\newcommand{\Z}{\mathbb{Z}}
\newcommand{\R}{\mathbb{R}}
\newcommand{\dof}{\bf\boldmath}

\begin{document}
\maketitle

\begin{abstract}
Holroyd and Propp used Hall's marriage theorem to show that,
given a probability distribution $\pi$ on a finite set $\calS$,
there exists an infinite sequence $s_1,s_2,\dots$ in $\calS$
such that for all integers $k \geq 1$ and all $s$ in $\calS$,
the number of $i$ in $[1,k]$ with $s_i = s$ differs from $k \,
\pi(s)$ by at most 1. We prove a generalization of this result
using a simple explicit algorithm. A special case of this
algorithm yields an extension of Holroyd and Propp's result to
the case of discrete probability distributions on infinite
sets.  [Note added in 2010: Since writing and posting this
article on the arXiv, we have learned that both Theorem 1
and Theorem 2 are in the literature; see the articles
by Tijdeman that have been added to the bibliography.]
\end{abstract}

\renewcommand{\thefootnote}{}
\footnotetext{{\bf\hspace{-6mm}Key words:} rotor router;
discrepancy; quasi-random sequence; low-discrepancy sequence}
\footnotetext{{\bf\hspace{-6mm}2000 Mathematics Subject
Classifications:} 82C20; 20K01; 05C25}
\footnotetext{\hspace{-6mm}OA funded in part by NSERC.
  AEH funded in part by Microsoft Research and NSERC.
  JBM supported by an EPSRC Advanced Fellowship. 
  JP funded in part by an NSF grant.}
\renewcommand{\thefootnote}{\arabic{footnote}}

Recently there has been an upsurge of interest in non-random
processes that mimic interesting aspects of random processes,
where the fidelity of the mimicry is a consequence of
discrepancy constraints built into the constructions (for a
general survey of discrepancy theory, see~\cite{chazelle}).  A
recent example is the work of Friedrich, Gairing and
Sauerwald~\cite{friedrich-gairing-sauerwald} on load-balancing;
other examples, linked by their use of the ``rotor-router
mechanism'', are the work of Cooper, Doerr, Friedrich, Spencer,
and Tardos
~\cite{cooper-doerr-friedrich-spencer,cooper-doerr-spencer-tardos-0,
cooper-doerr-spencer-tardos,cooper-spencer,doerr-friedrich} on
derandomized random walk on grids (``$P$-machines''), the work
of Landau, Levine and Peres
~\cite{landau-levine,levine-peres-3,levine-peres,levine-peres-2}
on derandomized internal diffusion-limited aggregation on grids
and trees, and the work of Holroyd and
Propp~\cite{holroyd-propp} on derandomized Markov chains. Here
we focus on derandomizing something even more fundamental to
probability theory: the notion of an independent sequence of
discrete random variables.

Given a discrete probability distribution $\pi(\cdot)$ on a set
$\calS$, the associated i.i.d.\ process satisfies the law of
large numbers.  That is, if we choose $S_1,S_2,\dots$ from
$\calS$ independently at random in accordance with $\pi$, the
random variables $N_k (s) := \#\{i: 1 \leq i \leq k \
\mbox{and} \ S_i = s\}$ have the property that $N_k (s) / k
\rightarrow \pi(s)$ almost surely as $k\to\infty$, and indeed
the discrepancies $N_k (s) - k \, \pi(s)$ are typically
$O(\sqrt{k})$.  A derandomized analogue of an i.i.d.\ process
should have the property that $N_k (s) - k \, \pi(s)$ is
$o(k)$, and derandomized processes with $|N_k (s) - k \,
\pi(s)|$ as small as possible are especially interesting. It is
too much to ask that the unscaled differences $N_k (s) - k \,
\pi(s)$ themselves go to zero (since $N_k (s)$ is always an
integer while $k \, \pi (s)$ typically is not), but we can ask
that these differences stay bounded.

Holroyd and Propp \cite{holroyd-propp} used such derandomized
i.i.d.\ processes (``low-discrepancy stacks'', in their
terminology) in order to make their theory applicable to Markov
chains with irrational transition probabilities. Indeed, the
following theorem appears (with slightly different notation) as
Proposition 11 in \cite{holroyd-propp}.

\begin{thm}[Low-discrepancy sequences for i.i.d.\
processes; \cite{holroyd-propp}]\label{hp} Given a probability
distribution $\pi$ on a finite set $\calS$, there exists an
infinite sequence $s_1,s_2,\dots$ in $\calS$ such that for all
$k \geq 1$ and all $s$ in $\calS$, the number of $i$ in $[1,k]$
with $s_i = s$ differs from $k \, \pi(s)$ by at most 1.
\end{thm}

Here we give a proof of this result that is simultaneously
simpler and more constructive than Holroyd and Propp's, and
applies even when the set $\calS$ is infinite. Furthermore, our
construction gives a simple way to derandomize sequences of
discrete random variables that are independent but not
identically distributed.

\begin{thm}[Low-discrepancy sequences for independent
processes]\label{main} Given discrete probability distributions
$\pi_1, \pi_2, \dots$ on some countable set
$\calS$, there exists an infinite sequence $s_1, s_2, \dots$ in
$\calS$ such that for all $k \geq 1$ and all $s$ in $\calS$,
the quantities $N_k (s) := \#\{i: 1 \leq i \leq k \text{ and }
\ s_i = s\}$ and $P_k (s) := \sum_{i=1}^k \pi_i (s)$ differ in
absolute value by strictly less than 1.
\end{thm}

Theorem \ref{hp} is the special case of Theorem \ref{main} in
which $\calS$ is finite and $\pi_1 = \pi_2 = \dots = \pi$, with
an infinitesimally weaker inequality in the conclusion.

\begin{remark*}
  \sloppypar For every choice of $s_1,s_2,\dots$, we have $\sum_s N_k (s) =
  k = \sum_s P_k (s)$, whence $\sum_s (N_k (s) - P_k (s)) = 0$. Moreover,
  if we were to choose $S_1,S_2,\dots$ independently from $\calS$ in
  accordance with the respective probability distributions $\pi_i$, then
  the expected number of $i$ in $[1,k]$ with $S_i = s$ would be
  $\sum_{i=1}^k \pi_i (s)$, so the expected value of $N_k (s) - P_k (s)$
  would be zero for each $s$.
\end{remark*}

\noindent
{\sc Note:} It has come to our attention that Theorem 2 is not new;
the same result (with the same proof) was discovered 
by Tijdeman in the 1970s (see~\cite{T1} and~\cite{T2}).
Indeed, Tijdeman's result is slightly stronger than ours
in the case where $\calS$ is finite (he obtains the
optimal constant, which is slightly smaller than 1
and depends on the size of $\calS$).

\begin{proof}[Proof of Theorem \ref{main}]
We present an algorithm for determining the sequence $(s_k)$.  
Our algorithm is as follows. Given $s_1,\dots,s_k$ (with $k
\geq 0$), let $s_{k+1}$ be the candidate $s$ with the earliest
deadline, where we say $s$ is a {\dof candidate} (for being the
$(k+1)$st term) if $N_k (s) - P_{k+1} (s) < 0$, and where we
define the {\dof deadline} for such an $s$ as the smallest
integer $k' \geq k+1$ for which $N_k (s) - P_{k'} (s) \leq -1$.
Ties may be resolved in any fashion.

For $k \geq 0$ write $D_k(s) := N_k(s) - P_k(s)$ (note that
$D_0(s)=0$).  First observe that $s$ is a candidate if and only
if taking $s_{k+1} = s$ would lead to $D_{k+1}(s) < 1$;
equivalently, $s$ fails to be a candidate if and only if taking
$s_{k+1} = s$ would lead to $D_{k+1}(s) \geq 1$. That is, when
we are choosing the $(k+1)$st term of the sequence, $s$ fails
to be a candidate if and only if choosing $s_{k+1}$ to be $s$
would cause $s$ to be {\dof oversampled} from time 1 to time
$k+1$. It is clear that there is always at least one candidate,
since $\sum_s (N_{k} (s) - P_{k+1} (s)) = -1 < 0$.

Also note that if $s$ is a candidate with deadline $k'$, then
taking $s_{k+1},\dots,s_{k'}$ all unequal to $s$ would lead to
$D_{k'}(s) \leq -1$; that is, such an $s$ would be {\dof
undersampled} from time 1 to time $k'$ if it were not chosen to
be at least one of $s_{k+1},\dots,s_{k'}$.

For $k' > k \geq 0$ define $R_{k,k'}(s) := \lfloor P_{k'} (s) -
N_k (s) \rfloor ^+$ (where $x^+ := \max(x,0)$). Thus
$R_{k,k'}(s)$ is the minimal number of the terms $s_{k+1},\dots,s_{k'}$
that must be equal to $s$ in order to prevent $s$ from being
undersampled from time 1 to time $k'$. If $R_{k,k+1}(s) = 1$,
then $s_{k+1}$ must be chosen to equal $s$ in order to prevent
$s$ from being undersampled from time 1 to time $k+1$; we call
such an $s$ {\dof critical}. We will show by induction on
$k\geq 0$ that $\sum_s R_{k,k'}(s) \leq k'-k$ for all $k'>k$.
This implies in particular that $\sum_s R_{k,k+1}(s) \leq 1$
for all $k$, so that at each step at most one $s$ is critical;
this in turn implies that no $s$ is ever undersampled. And
since our procedure only chooses candidates, no $s$ is ever
oversampled either.

First, consider $k=0$: we have $D_0 (s) = 0$ and $R_{0,k'}(s) =
\lfloor P_{k'} (s) \rfloor$, so $\sum_s R_{0,k'}(s) = \sum_s
\lfloor P_{k'} (s) \rfloor \leq \lfloor \sum_s P_{k'} (s)
\rfloor = k' - 0$ as claimed.

Now take $k\geq 0$, and suppose for induction that $\sum_s
R_{k,k'}(s) \leq k'-k$ for some particular $k' > k+1$. We wish 
to show that $\sum_s R_{k+1,k'}(s) \leq k'-(k+1)$. There are 
two cases to consider.
First, if $R_{k,k'}(s_{k+1}) > 0$, then $R_{k+1,k'}(s_{k+1}) =
R_{k,k'}(s_{k+1})-1$ and $R_{k+1,k'}(s) = R_{k,k'}(s)$ for all
$s \neq s_{k+1}$, so $\sum_s R_{k+1,k'}(s) = (\sum_s
R_{k,k'}(s) ) - 1 \leq (k'-k) - 1 = k'-(k+1)$ as claimed.
Second, if $R_{k,k'}(s_{k+1}) = 0$, then the deadline for
$s_{k+1}$ is greater than $k'$. Since our algorithm chooses
$s_{k+1}$ as the $(k+1)$st term, $s_{k+1}$ must be the candidate
with the earliest deadline. This means that no $s \neq s_{k+1}$
in $\calS$ with $R_{k,k'} (s) > 0$ is a candidate; that is,
every $s \neq s_{k+1}$ with $R_{k,k'} (s) > 0$ must satisfy
$N_{k} (s) \geq P_{k+1} (s)$, and since $N_{k+1} (s) = N_{k}
(s)$, we must have $D_{k+1}(s) \geq 0$, implying $R_{k+1,k'}(s)
= \lfloor -D_{k+1}(s) + \sum_{i=k+2}^{k'} \pi_i(s) \rfloor ^+
\leq \lfloor \sum_{i=k+2}^{k'} \pi_i(s) \rfloor ^+  \leq \sum_{i=k+2}^{k'}
\pi_i(s)$. Likewise, for all $s$ with $R_{k,k'}(s) = 0$, we
have $R_{k+1,k'}(s) = 0$, implying $R_{k+1,k'}(s) \leq
\sum_{i=k+2}^{k'} \pi_i(s)$. Hence $\sum_s R_{k+1,k'}(s) \leq
\sum_s \sum_{i=k+2}^{k'} \pi_i(s) = \sum_{i=k+2}^{k'} \sum_s
\pi_i(s) = \sum_{i=k+2}^{k'} 1 = k'-(k+1)$, as claimed.
\end{proof}

In reading the proof, the reader may find it helpful to imagine
the following scenario. Let $\calS$ be a set of creatures, each
of which has a {\dof surplus} (or ``energy level'') that is
initially 0. At each step a single creature gets fed. At the
$k$th step, the surplus $D_k(s)$ of creature $s$ decreases by
$\pi_k(s)$, but in addition increases by $1$ (giving a net
change of $1 - \pi_k (s)$) if $s$ gets fed. After each step the
sum of the surpluses is zero. If a creature's surplus ever
falls to $-1$ or less, the creature dies of starvation; if its
surplus ever rises to $1$ or more, it dies of overfeeding. Our
strategy for keeping all the creatures alive is to always feed
the creature that if left unfed would die earliest of
starvation, excepting those that cannot be fed because they
would immediately die of overfeeding.

Independently of our work, in the context of a one-sided
version of the discrepancy-control problem arising from an
email post by John Lee~\cite{lee}, Oded Schramm and Fedja
Nazarov considered other algorithms for keeping the
quantities $k \, \pi(s) - N_k(s)$ from becoming too large in
the case where all the $\pi_i$ equal $\pi$. 
% Schramm considered the algorithm that at each
% stage picks the $s$ with the earliest deadline, and showed that
% for this scheme, $k \, \pi(s) - N_k(s)$ never exceeds $1$.
% This does not solve our problem since it permits unbounded
% negative values of $k \, \pi(s) - N_k(s)$.

Considering the sequence of discrepancy vectors $(D_k)$ leads
to the following reformulation of Theorem~\ref{hp}.
\begin{cor}\label{alt}
For any probability vector $\pi \in \R^n$ there is a compact
$K\subseteq[-1,1]^n$ containing $(0,0,\dots,0)$ such that
$K\subseteq \bigcup_{i=1}^n (K + \pi - e_i)$, where
$\{e_1,\ldots,e_n\}$ is the standard basis of $\R^n$.
\end{cor}
To see that this implies Theorem \ref{hp}, note that, given $K$
and $D_k$ (with $k \geq 0$), we can choose $s_{k+1}=i$ such
that $D_k - \pi + e_i \in K$; this choice guarantees that the
discrepancy vector $D_{k+1}$ lies in $K$ and hence in
$[-1,1]^n$. Conversely, note that if we take the sequence $s_1,
s_2, \dots$ given by Theorem \ref{hp}, the bounded set $\{D_k:
\ k \geq 0\}$ satisfies all the conditions of Corollary
\ref{alt} except for compactness. Hence we can prove Corollary
\ref{alt} by taking $K$ to be the closure of $\{D_k: \ k \geq
0\}$.  Corollary \ref{alt} asserts the existence of a set $K$
containing the origin that is covered by translations of itself
by given vectors $v_i:=\pi-e_i$.  Since the zero vector is in
the convex hull of the $v_i$, a sufficiently large ball in the
subspace spanned by the $v_i$ achieves this, but Corollary
\ref{alt} provides the bound $K \subseteq [-1,1]^n$.

The constant 1 in Theorem \ref{main} cannot be improved;
that is, there is no $c < 1$
with the property that for all $\pi_1, \pi_2, \dots$ there
exists a way to choose $S_1,S_2,\dots$ from $\calS$ so that
the discrepancies $N_k (s) / k - \pi(s)$ all stay
within the interval $(-c,+c)$.
Consider for instance the case where
each $\pi_i$ is the uniform distribution on a finite set $\calS$ of
cardinality $n$, and take $k = n-1$.
There exists $s$ in $\calS$ distinct from $s_1,\dots,s_k$
and this $s$ satisfies $P_k (s) - N_k (s) = k/n - 0 = 1-1/n$.
Although this example shows that the constant 1 cannot be improved,
it is possible that there is a universal
strict subset $K$ of $(-1,+1) \times (-1,+1) \times \cdots$
with the property that for all $\pi_1, \pi_2, \dots$ there
exists a way of choosing $S_1,S_2,\dots$ from $\calS$ so that
the vector of discrepancies stays within the set $K$.

It is also possible that Theorem \ref{main} might be strengthened by
controlling the discrepancies between $N_k(s) - N_j(s)$ and
$\pi_{j+1} (s) + \pi_{j+2} (s) + \cdots + \pi_{k} (s)$ for all
$j,k$ with $1 \leq j \leq k$ and all $s$ in $\calS$. It is easy
to deduce from Theorem \ref{main} that every such discrepancy
has absolute value less than 2 (since it is just $D_k(s) - D_j(s)$),
but perhaps one can show that there exists a way of choosing $S_1,S_2,\dots$ 
so that every such discrepancy has absolute value less than 1.

In determining $s_k$, our algorithm typically requires knowledge 
of the future distributions $(\pi_i)_{i\geq k}$ 
(actually a finite but unbounded number of them).  
This is unavoidable, as may be seen by the following example.  
Let $\pi_k$ be uniform on $1,\ldots,5$ for each of $k=1,2,3$. 
Regardless of $s_1,s_2,s_3$ there will be two $i$'s with $N_3(i)-P_3(i)=-0.6$.
If $s_1,s_2,s_3$ are chosen in ignorance of $\pi_4$, 
it is possible for $\pi_4$ to be uniform on those two $i$'s.
Then any choice of $s_4$ will result in $N_4(i)-P_4(i)=-1.1$ for some $i$. 
With more than $5$ values the discrepancies can be even larger in magnitude.
It might be interesting to know how good a bound on discrepancy
can be achieved by algorithms that are constrained to have
$s_n$ depend only on $\pi_1,\dots,\pi_n$.
Of course, when the $\pi_i$'s are all equal as in Theorem~\ref{hp} 
this issue is nonexistent.

One way in which Theorem \ref{hp} might be strengthened is by
finding a construction that minimizes $\max_k \sum_{i<k}
f(s_i)$, where $f$ is some function on $\calS$ satisfying
$\sum_s f(s) \pi(s) = 0$. Sums of the form $\max_k \sum_{i<k}
f(s_i)$ play an important role in~\cite{holroyd-propp}; there
the elements $s$ of $\calS$ correspond to transitions $u
\rightarrow v$ in a Markov chain (with $u$ fixed and $v$
varying), $\pi(s)$ equals the transition probability $p(u,v)$,
and $f(s)$ equals $h(v)-h(u)$ where the function $h(\cdot)$ is
harmonic at $u$ (i.e., $\sum_v p(u,v) h(v) = h(u)$), implying
$\sum_s f(s) \pi(s) = 0$. For example, consider random walk on
$\Z^2$, where a vertex $u$ has four neighbors $u_N$, $u_S$,
$u_E$, and $u_W$ with $p(u_N) = p(u_S) = p(u_E) = p(u_W) =
\frac14$.  Key results in~\cite{holroyd-propp} treat
rotor-routers that rotate in the repeating pattern
$N,E,S,W,N,E,S,W,\dots$ and show that the resulting
rotor-router walks closely mimic certain features of the random
walk. However, since the relevant discrepancies are controlled
by quantities of the form $\max_k \sum_{i<k} f(s_i)$, and since
the function $h$ has the property that $h(u_N) + h(u_S)$ and
$h(u_E) + h(u_W)$ are close to $2h(u)$ (implying that $f(N) +
f(S)$ and $f(E) + f(W)$ are close to zero), there is reason to
think that rotor-routers that rotate in the repeating pattern
$N,S,E,W,N,S,E,W,\dots$ would give smaller discrepancy for the
quantities of interest.

As an important special case of Theorem \ref{hp}, 
suppose $\pi$ is rational. It is then
natural to ask that the sequence $s_1,s_2,\ldots$ be periodic
(so that one has a ``rotor'' in the sense
of~\cite{holroyd-propp}).  Our algorithm as described does not
guarantee periodicity, because we allowed ties to be broken
arbitrarily. However, if we add the stipulation that ties are
always broken in some pre-determined way depending only on the
deadlines and the discrepancies $D_k(s)$, then our algorithm
yields a periodic sequence, with period equal to the least
common multiple $m$ of the denominators of the rational numbers
$\pi(s)$.  Indeed, it is clear that the sequence generated
by the algorithm is eventually periodic,
and that the period cannot be less than $m$.
On the other hand, the construction gives $|N_m(s)-P_m(s)| <
1$; but $N_m(s)$ and $P_m(s)$ are both integers, so they must
be equal. Hence $D_m(s) = 0 = D_0(s)$ for all $s$, so the
procedure enters a loop at time $m$.

Theorem \ref{hp} can be rephrased as follows: for any sequence
of non-negative real numbers $\pi(1),\pi(2), \dots$ summing to
1, there exists a partition of the natural numbers into sets
$F_1,F_2,\dots$ where for all $i$ the set $F_i$ has density
$\pi(i)$, and $|F_i \cap \{1,\dots,k\}| - \pi(i) k$ lies in
$[-1,1]$ for all $k \geq 1$ (of course the second property
implies the first). An even stronger condition we might seek is
that the gap between the $m$th and $n$th elements of $F_i$ (for
all $i$ and all $n \geq m \geq 1$) is within 1 of
$(n-m)/\pi(i)$. If ``within 1'' is interpreted in the strict
sense (i.e., the difference is strictly less than 1), then this
condition cannot always be achieved; e.g., with
$\pi=(\frac12,\frac13,\frac16)$, the only way to satisfy the
condition would be to partition the natural numbers into three
arithmetic progressions with densities $\frac12$, $\frac13$,
and $\frac16$, which clearly cannot be done. However, if
``within 1'' is interpreted in the weak sense (i.e., the
difference is less than or equal to 1), then we do not know of
a counterexample.

\section*{Acknowledgments}

We thank the Pacific Institute for the Mathematical Sciences
and the Centre de Recherches Math\'ematiques for bringing the
authors together at the workshop {\it New Directions in Random
Spatial Processes} in May 2009, thereby making this research
possible.  We thank Fedja Nazarov, Oded Schramm, and Joel
Spencer for valuable discussions.

\bibliography{ahmp}

\bigskip

\noindent
{\sc Omer Angel:}
{\tt angel at math dot ubc dot ca} \\
{\sloppy Dept.\ of Mathematics, University of British Columbia,
Vancouver, BC, Canada.}

\vspace{3mm}
\noindent
{\sc Alexander E.\ Holroyd:}
{\tt holroyd at math dot ubc dot ca}\\
Microsoft Research, 1 Microsoft Way, Redmond WA, USA; and
\\ University of British Columbia, 121-1984 Mathematics Road,
Vancouver, BC, Canada.

\vspace{3mm}
\noindent
{\sc James B.\ Martin:}
{\tt martin at stats dot ox dot ac dot uk}\\
Department of Statistics, University of Oxford, 1 South Parks Road, Oxford, UK.

\vspace{3mm}
\noindent
{\sc James Propp:}
{\tt jpropp at cs dot uml dot edu}\\
University of Massachussets Lowell,
1 University Ave., Olney Hall 428, Lowell, MA,  USA.

\end{document}